\documentclass{amsart}

\usepackage{amsmath}
\usepackage[alphabetic]{amsrefs}
\usepackage{hyperref}
\usepackage{mathrsfs}
\usepackage{tikz-cd}

\newtheorem{theorem}{Theorem}[section]
\newtheorem{lemma}[theorem]{Lemma}
\newtheorem{proposition}[theorem]{Proposition}
\newtheorem{corollary}[theorem]{Corollary}

\theoremstyle{definition}
\newtheorem{definition}[theorem]{Definition}

\theoremstyle{remark}
\newtheorem{remark}[theorem]{Remark}

\numberwithin{equation}{section}

\def\Spec{\mathop{\mathrm{Spec}}}

\def\F{{\mathscr{F}}}
\def\Fu{\underline{\F}}
\def\Fun{\underline{\F}_{\leq{n}}}
\def\O{\mathscr{O}}
\def\WO{{W\Omega_X^N}}

\def\WOn{{W_n\Omega_X^N}}

\def\WOno{{W_{n+1}\Omega_X^N}}

\def\fa{{\text{~for any~}}}

\def\tensor{\mathop{\otimes}\limits}
\def\dsum{\mathop{\bigoplus}}
\def\sHom{\mathop{\mathcal{H}\! \mathit{om}}\nolimits}
\def\Hom{\mathop{\mathrm{Hom}}\nolimits}

\makeatletter
\newcommand{\proofstep}[1]{%
  \par
  \addvspace{\medskipamount}
  \textit{#1\@addpunct{:}}\enspace\ignorespaces
}
\makeatother

\begin{document}

\title{Duality for Witt--divisorial sheaves}


\author{Niklas Lemcke}

\address{Department of Mathematics, School of Science and Engineering, Waseda University,
Ohkubo 3-4-1, Shinjuku, Tokyo 169-8555, Japan}
\email{numberjedi@akane.waseda.jp}

\begin{abstract}
    We adapt ideas from Ekedahl~\cite{Eke} to prove a Serre-type duality for Witt-divisorial sheaves of $\mathbb Q$--Cartier divisors on a smooth projective variety over a perfect field of finite characteristic. We also explain its relationship to Tanaka's vanishing theorems~\cite{Tanaka}.
\end{abstract}

\maketitle

\tableofcontents

\section*{Introduction}
    Kodaira Vanishing and its generalizations have been crucial in the development of the Minimal Model Program (MMP) in characteristic zero. However, as is well-known, they do not hold in positive characteristic. Tanaka~\cite{Tanaka} proposed a Kodaira-like vanishing theorem which holds for ample divisors in positive characteristic. 

    \begin{theorem}[\cite{Tanaka}, cf. Theorem~\ref{thmTanaka}]\label{thm0}
        Let $k$ be a perfect field of characteristic $p>0$, and $X \xrightarrow \phi \text{Spec }k$ be an $N$--dimensional smooth projective variety.
        If $A$ is an ample $\mathbb Q$--Cartier divisor on $X$, then
        \renewcommand{\theenumi}{\roman{enumi}}
        \begin{enumerate}
            \item $H^j(X, W\O_X(-A)) = p^t\text{--torsion}, \text{ for some } t, \fa j < N$.
            \item 
                \begin{itemize}
                    \item $R^i\phi_*\sHom_{W\O_X}(W\O_X(-A),\WO)_\mathbb Q = 0$
                    \item $H^i(X, W\Omega_X^N \tensor_{W\O_X} W\O_X(A)) = 0 \fa i>0,$ if $A$ is Cartier.
                \end{itemize}
        \end{enumerate}
    \end{theorem}
    Interestingly, the proof of (i) is easier than the proof of (ii).
    Ideally we would want the theorem to hold for nef and big invertible sheaves, but this is not yet known.
    The purpose of this paper is to establish a duality property for the Witt--divisorial sheaf $W\O_X(D)$ associated to a $\mathbb Q$--Cartier divisor $D$ on $X$. 

    In~\cite{Eke}, Ekedahl introduces a duality functor $D$, and eventually constructs an isomorphism (\cite{Eke}*{Theorem III: 2.9})
    \[
         D(R\Gamma(W\Omega_X^\bullet))(-N)[-N] \cong R\Gamma(W\Omega_X^\bullet),
    \]
    where $(-N)$ and $[-N]$ denote shifts in module and complex degree, respectively.
    He then shows that
    \[
        D(R\Gamma(W\Omega_X^\bullet)) \cong R\Hom_R(R\Gamma(W\Omega_X^\bullet), \check{R}),
    \]
    in $D(R)$, where $R$ is the Raynaud Ring (a non-commutative $W$-algebra), and $\check R$ is a certain $R$--module. 
    Where Ekedahl uses the Raynaud ring $R$, we use the similar Cartier-Dieudonn\'e-ring $W[F,V] =: \omega $.
    \begin{theorem}[Cf. Theorem~\ref{thmDuality}]\label{thm1}
        Let $X$ be a smooth projective variety over a perfect field $k$ of characteristic $p > 0$, and $D$ be a $\mathbb Q$--Cartier divisor on $X$.
        Then
        \[
            \begin{aligned}
                \prod_{t\in\mathbb Z} R\phi_*R\lim_nR\sHom_{W_n\O_X}(W_n\O_X(p^tD),\WOn) \\
                \cong R\Hom_\omega\left(\dsum_{t\in\mathbb Z} R\phi_* W\O_X(p^tD), \check\omega[-N]\right),
            \end{aligned}
        \]
        for a certain left--$\omega$--module $\check\omega$.
    \end{theorem}
    This allows us to recover Tanaka's vanishing theorem, as well as to make the (possibly) non--vanishing torsion somewhat more explicit.
    \subsection*{Acknowledgements}
            I want to thank the referee for many insightful comments, in particular suggesting a much simpler proof for the vanishing of the higher derived limits in Proposition~\ref{propMiniDuality}.
            Further, I am particularly grateful to Professor Tanaka Hiromu for agreeing to meet me and answer my questions regarding his work, as well as my advisor Professor Kaji Hajime for his unwavering support and exacting attention to detail during our discussions.

\section{Notation}
    Fix the following notations and conventions:
    \begin{itemize}
        \item A variety over $k$ is a separated integral scheme of finite type over $k$.
        \item Throughout this paper we define $X \xrightarrow \phi S=\Spec k$, where $k$ is a perfect field of characteristic $p>0$, and $X$ is assumed to be a smooth projective variety.
        \item If $C$ is a complex, $C[i]$ denotes $C$ shifted by $i$ in complex degree.
        \item If $M_n$ is an inverse system, then $\lim_n M_n$ denotes the inverse limit.
        \item For a module $M$, we write $M_\mathbb Q := M \tensor_\mathbb Z \mathbb Q$.
    \end{itemize}

\section{Preliminaries}
    This section serves to recall some definitions and known results. 

\subsection{Tanaka's vanishing}
    The original Kodaira Vanishing is closely related to Hodge decomposition. 
	Hodge decomposition in turn resembles the slope decomposition of crystalline cohomology in terms of the de Rham-Witt complex. 
	This motivates the attempt at finding a useful vanishing theorem in the context of de Rham-Witt.
    \begin{definition}[Teichm\"uller lifts of line bundles, cf.~\cite{Tanaka}]
        For a ring $A$, any element $a\in A$ can be naturally identified with an element in $W(A)$ by
        \[
            A \rightarrow W(A)
        \]
        \[
            a \mapsto \underline a := (a,0,0,\cdots).
        \]
        This $\underline a$ is called the {\it Teichm\"uller representative} of $a$.
        For an invertible sheaf $\mathscr F$ on $X$ defined by local transition functions $(f_{ij})$, Tanaka defines the {\it Teich\-m\"uller lift} $\underline {\mathscr F}$ of an invertible $\O_X$--module to be the invertible $W\O_X$--module given by the Teichm\"uller representatives of the transition functions $(\underline {f_{ij}})$.
        The {\it truncated} Teichm\"uller lift is defined by
        \[
            \Fun := W_n\O_X \tensor_{W\O_X} \Fu.
        \]
    \end{definition}

    \begin{definition}[Witt-divisorial sheaves, cf.~\cite{Tanaka}]\label{defWOD}
        Alternatively, the {\it Witt-divisorial sheaf associated to} an $\mathbb R$--divisor $D$ is defined by
        \[
            \Gamma_V(W\O_X(D)) := \left\{ (\phi_0, \phi_1, \cdots) \in W(K(X)); \text{div}(\phi_n) + p^n D|_V \geq 0 \right\}.
        \]
        As Tanaka shows, for a Cartier divisor on a reasonably nice scheme, these two notions are equivalent, since $W\O(D)|_U = \underline{f} W\O_X|_U$ for $U$ affine open in $X$ and $f$ a local equation for $D$ on $U$ (cf.~\cite{Tanaka}*{Proposition 3.12}). 
        $W_n\O_X(D)$ is a coherent $W_n\O_X$--module (cf. \cite{Tanaka}*{Proposition 3.8}).
    \end{definition}

    The following two propositions due to Tanaka~\cite{Tanaka} will be used frequently throughout this paper, often without explicit reference.
    
    \begin{proposition}[Cf. \cite{Tanaka}*{Proposition 3.15}]
        Let $D$ be an $\mathbb R$--divisor on $X$. Then, for any $0 \leq e, 0<m\leq n$, there is an isomorphism
        \[
            R\sHom_{W_n\O_X}((F^e)_*W_m\O_X(D), \WOn)
            \cong (F^e)_*\sHom_{W_m\O_X}(W_m\O_X(D), W_m\Omega_X^N)
        \]
        in $D(W\O_X-\mathfrak{mod})$.
    \end{proposition}

    \begin{proposition}[Cf. \cite{Tanaka}*{Proposition 4.9 and Lemma 2.10}]
        Let $D$ be an $\mathbb R$--divisor on $X$. 
        Let $M$ be a coherent $W_n\O_X$--module such that the induced map $M(U) \rightarrow M_\xi$ is injective for any non--empty open subset $U \subset X$, where $M_\xi$ denotes the stalk of $M$ at the generic point $\xi$ of $X$.
        Then the induced $W\O_X$--module homomorphism
        \[
            \sHom_{W_n\O_X}(W_n\O_X(D),M) \xrightarrow \theta
            \sHom_{W\O_X}(W\O_X(D),M)
        \]
        is an isomorphism.

        $\WOn$ and $\text{gr}^n\WO$ are two such $W_n\O_X$--modules.
    \end{proposition}

    \begin{theorem}[Tanaka, {cf.\ \cite{Tanaka}*{Theorem 1.1}}]\label{thmTanaka}
        Let $k$ be a perfect field of characteristic $p>0$, and $X$ be an $N$--dimensional smooth projective variety over $k$.
        If $A$ is an ample $\mathbb Q$-Cartier divisor on $X$, then there exists $ s_0$ such that for all $s_0 < s$,
        \renewcommand{\theenumi}{\roman{enumi}}
        \begin{enumerate}
            \item 
                \begin{itemize}
                    \item $H^j(X, W_n\O_X(-sA)) = 0 \fa j < N, n \in \mathbb N$,
                    \item $H^j(X, W\O_X(-sA)) = 0 \fa j < N$,
                    \item $H^j(X, W\O_X(-A)) = p^t\text{--torsion, for some } t, \fa j < N$.
                \end{itemize}
            \item 
                \begin{itemize}
                    \item $R^i\phi_*\sHom_{W\O_X}(W\O_X(-A),\WO)_\mathbb Q = 0 \fa 0<i$,
                    \item $R^i\phi_*(W_n\O_X(sA)\tensor\WOn) = 0 \fa 0<i, n \in \mathbb N, A$ Cartier,
                    \item $R^i\phi_*(W\O_X(sA)\tensor\WO) = 0 \fa 0<i, A$ Cartier,
                    \item $R^i\phi_*(W\O_X(A)\tensor\WO ) = 0 \fa 0<i, A$ Cartier.
                \end{itemize}
        \end{enumerate}
    \end{theorem}
        
	\begin{remark}
        The theorem appears to suggest a Serre-type duality.
        This duality would be asymmetric in the sense that torsion from (i) does not appear in (ii).
		Note that the proof of (i) is simple relative to that of (ii).
		So ideally duality would recover (ii) from (i), potentially facilitating the proof of the theorem for nef and big $D$.
	\end{remark}

\section{Duality}
\subsection{Duality Theorem}
    \begin{proposition}\label{prop1}
        Let $\F$ be an invertible $\O_X$--module.
        For any $n>0$, 
        \begin{equation}\label{prop1RelEq}
            W_n\Omega_X^N\tensor_{W_n\O_X}\Fun \cong R\sHom_{W_n\O_X}(\Fun^{\vee}, \WOn)
        \end{equation}
        such that, in particular, 
        \begin{equation}\label{prop1AbsEq}
            {H^i(X,\WOn\tensor_{W_n\O_X}\Fun)} 
            \cong 
            \Hom_{W_n}(H^{N-i}(X,{\Fun^\vee}), W_n)
            \fa i\geq0, n>0.
        \end{equation}
    \end{proposition}
    \begin{proof}
        We have
        \[
            \begin{aligned}
                \WOn\tensor\Fun 
                &\cong \sHom_{W_n\O_X}(W_n\O_X, \WOn\tensor\Fun) \\
                &\cong \sHom_{W_n\O_X}(\Fun^\vee, \WOn),
            \end{aligned}
        \]
        where the second isomorphism holds because $\Fun$ is locally free, and so $-\tensor\Fun^\vee$ is fully faithful.
        Since $R^i\sHom(\Fun^\vee, \WOn) = 0 \fa 0 < i$ (by local freeness), Equation~\ref{prop1RelEq} holds.
        To show Equation~\ref{prop1AbsEq}, take global sections of the derived push-forward.
        \[
            \begin{aligned}
                \Gamma_S(R\phi_*(\WOn\tensor_{W_n\O_X}\Fun))
                &\cong \Gamma_S(R\phi_*R\sHom_{W_n\O_X}(\Fun^{\vee}, \WOn)) \\
                &\cong \Gamma_S(R\sHom_{W_n\O_S\cong W_n}(R\phi_*(\Fun^{\vee}), W_n[-N])) \\
                &\cong \Hom_{W_n}((R\phi_*\Fun^{\vee})[N], W_n),
            \end{aligned}
        \]
        where $W_n$ is the constant sheaf, the second isomorphism is due to Coherent Duality and~\cite{Eke}*{I, Theorem 4.1}, and the third isomorphism is due to $W_n$ being an injective $W_n$--module.
        In particular for all $i$ there are isomorphisms
        \[
            H^i(X, \WOn\tensor_{W_n\O_X}\Fun) 
            \cong \Hom_{W_n}(H^{N-i}(X, \Fun^{\vee}), W_n).
        \]
    \end{proof}

    We now attempt passing to the limit. 
    First, recall the following result.

    \begin{lemma}[Chatzistamatiou, R\"ulling, cf. \cite{CR}*{Lemma 1.5.1}]\label{lemCR}
        Let $(X,\O_X)$ be a ringed space and $E=(E_n)$ a projective system of $\O_X$--modules (indexed by integers $1\leq n$). Let $\mathscr B$ be a basis of the topology of $X$. We consider the following two conditions:
        \begin{enumerate}
            \item For all $U \in \mathscr B$, $H^i(U,E_n) = 0 \fa i,1\leq n$.
            \item For all $U \in \mathscr B$, the projective system $(H^0(U,E_n))_{n\geq 1}$ satisfies the Mittag--Leffler condition.
        \end{enumerate}
        Then
        \begin{itemize}
            \item If $E$ satisfies condition (1), then $R^i\lim_n E_n = 0 \fa 2 \leq i$.
            \item If $E$ satisfies conditions (1) and (2), then $R^i\lim_n E_n = 0 \fa 1\leq i$, i.e. $E$ is $\lim$--acyclic.
        \end{itemize}
    \end{lemma}

    \begin{lemma}\label{lemNToInfty}
        For $\F$ an invertible sheaf of $\O_X$--modules,
        \begin{equation}\label{eqLim}
            \WO\tensor_{W\O_X}\Fu \cong \lim_n(\WOn \tensor_{W_n\O_X}\Fun) \cong R\lim_n(\WOn \tensor_{W_n\O_X}\Fun).
        \end{equation}
    \end{lemma}
    \begin{proof}
        \[
            \begin{aligned}
                \WO\tensor_{W\O_X}\Fu
                &\cong \lim_n \WOn \tensor_{W\O_X}\Fu \\
                &\cong \lim_n (\WOn \tensor_{W\O_X}\Fu) \\
                &\cong \lim_n (\WOn \tensor_{W_n\O_X}\Fun).
            \end{aligned}
        \]
        Take the exact sequence (cf. \cite{Illusie}) of $W_{n+1}\O_X$--modules
        \[
            0 \rightarrow {\mathrm gr}^n\WO \rightarrow \WOno \rightarrow \WOn \rightarrow 0,
        \]
        where ${\mathrm gr}^n\WO$ is coherent.
        Tensoring with $\Fu$ over $W\O_X$ yields an exact sequence
        \[
            \begin{aligned}
                0 &\rightarrow {\mathrm gr}^n\WO\tensor_{W_{n+1}\O_X/p}\Fu_{\leq n+1}/p \\
                &\rightarrow \WOno\tensor_{W_{n+1}\O_X}\Fu_{\leq n+1} 
                \rightarrow \WOn\tensor_{W_n\O_X}\Fun \rightarrow 0.
            \end{aligned}
        \]
        For any $x \in X$, let $U_x$ be an affine open neighborhood of $x$. 
        Then 
        \[
            H^1(U_x, {\mathrm gr}^n\WO\tensor_{W_{n+1}\O_X/p}\Fu_{\leq n+1}/p) = 0
        \]
        by coherence, and therefore 
        \renewcommand{\theenumi}{\roman{enumi}}
        \begin{enumerate}
            \item
                $H^i(U_x, \WOn\tensor_{W_n\O_X}\Fun) = 0 \fa i>0,$ again by coherence,
            \item
                $H^0(U_x, \WOno\tensor_{W_{n+1}\O_X}\Fu_{\leq n+1}) \rightarrow H^0(U_x, \WOn\tensor_{W_n\O_X}\Fun)$
        is surjective for all $n>0$.
        \end{enumerate}
        In this fashion a basis $\mathscr U$ for the topology of $X$ can be found, such that the above two properties hold for all $U \in \mathscr U$, and so by Lemma~\ref{lemCR}, Equation~\ref{eqLim} holds.
    \end{proof}

    The following `twisting' lemma is the reason for the asymmetry between the vanishing theorems.
    \begin{lemma}[Twisting]\label{lemTanakaTwisting}
        Let $D$ be a $\mathbb Q$--Cartier divisor on $X$ such that $p^tD$ is a $\mathbb Z$--divisor for some positive integer $t$. 
        Then
        \[
            \sHom_{W\O_{X,\mathbb Q}}(W\O_X(D)_\mathbb Q,\WO_\mathbb Q)
            \cong F_*\sHom_{W\O_{X,\mathbb Q}}(W\O_X(pD)_\mathbb Q, \WO_\mathbb Q).
        \]
        Furthermore, if $D$ is $\mathbb Z$--Cartier, then
        \[
            \WO\tensor_{W\O_X}W\O_X(D)
            \cong (F)_*(\WO\tensor_{W\O_X}W\O_X(pD)).
        \]
    \end{lemma}
    \begin{proof}
        The proof is due to Tanaka, see for example \cite{Tanaka}*{Theorem 4.2 (4) and Theorem 4.13}. 
        It is repeated here for the reader's convenience.

        We have an induced isomorphism 
        \[
            (X,W\O_{X,\mathbb Q}) \xrightarrow{\substack{F \\ \sim}} (X,W\O_{X,\mathbb Q}).
        \]
        In particular, $F^*\circ F_* = F_* \circ F^* = 1$. We obtain the following chain of isomorphisms:
        \[
            \begin{aligned}
                \sHom_{W\O_{X,\mathbb Q}}(W\O_{X}(D)_\mathbb Q, \WO_{,\mathbb Q})
                &\cong F_*F^* \sHom_{W\O_{X,\mathbb Q}}(W\O_{X}(D)_\mathbb Q, \WO_{,\mathbb Q}) \\
                &\cong F_* \sHom_{F^*W\O_{X,\mathbb Q}}(F^*W\O_{X}(D)_\mathbb Q, F^*\WO_{,\mathbb Q}) \\
                &\cong F_* \sHom_{W\O_{X,\mathbb Q}}(W\O_{X}(pD)_\mathbb Q, F^*F_*\WO_{,\mathbb Q}) \\
                &\cong F_* \sHom_{W\O_{X,\mathbb Q}}(W\O_{X}(pD)_\mathbb Q, \WO_{,\mathbb Q}).
            \end{aligned}
        \]
        This proves the first statement.
        
        For the second statement, recall that the Frobenius homomorphism
        \[
            \WO \xrightarrow{F} F_*\WO
        \]
        is an isomorphism of $W\O_X$--modules (cf.  \cite{Tanaka}*{Theorem 2.9}).
        Therefore
        \[
            \begin{aligned}
                \WO\tensor_{W\O_X}W\O_X(D)
                &\cong F_*(\WO)\tensor_{W\O_X}W\O_X(D) \\
                &\cong F_*(\WO\tensor_{W\O_X}F^*W\O_X(D)) \\
                &\cong F_*(\WO\tensor_{W\O_X}W\O_X(pD)).
            \end{aligned}
        \]
        where the second isomorphism is the projection formula.

    \end{proof}

    We can now observe a first, tenuous duality between Theorem~\ref{thmTanaka}~(i) and (ii).
    \begin{proposition}\label{propMiniDuality}
        Let $D$ be a $\mathbb Q$--Cartier divisor on $X$ such that $p^tD$ is a $\mathbb Z$--divisor for some positive integer $t$. 
        Then 
        \[
            R^i\lim_n R^j\phi_*\sHom_{W_n\O_X}(W_n\O_X(D), \WOn) = 0 \fa 0<i, j \in \mathbb N.
        \]
        Suppose there exists $t$ such that
        \[
            H^j(X, W_n\O_X(p^tD)) = 0 \fa 0<n, j<N. 
        \]
        Then 
        \[
            R^i\phi_*\sHom_{W\O_X}(W\O_X(D), \WO)_\mathbb Q = 0 \fa 0 < i.
        \]
        If further D is Cartier, then 
        \[
            H^i(X, \WO\tensor_{W\O_X}W\O_X(-D)) = 0 \fa 0 < i.
        \]
    \end{proposition}
    \begin{proof}
        Set $E_n := \sHom_{W_n\O_S}(R^j\phi_*W_n\O_X(D),W_n)$. Since $S=\Spec k$, we have $H^i(S,E_n) = 0 \fa 0 < i$.
        By Lemma~\ref{lemCR} then $R^i\lim_n E_n = 0 \fa 1<i$.
        $W_nX$ is a proper scheme, so by coherence $H^j(X,W_n\O_X(D))$ and $E_n$ are finite, hence Artinian $W_n$--modules.
        But a projective system of Artinian $W_n$--modules satisfies the Mittag--Leffler condition, and so the first statement holds by Lemma~\ref{lemCR}, Coherent Duality and \cite{Eke}*{I, Theorem 4.1}.
        By the Twisting Lemma we have
        \[
            \begin{aligned}
                &R\phi_*\sHom_{W\O_{X,\mathbb Q}}(W\O_X(D)_\mathbb Q,\WO_\mathbb Q) \\
                &\cong (F^t_S)_*R\phi_*\sHom_{W\O_X}(W\O_X(p^tD),\WO)_\mathbb Q \\
                &\cong (F^t_S)_*\left(R\lim_n\Hom_{W_n}(H^N(X,W_n\O_X(p^tD)),W_n)\right)_\mathbb Q \\
                &\cong (F^t_S)_*\left(\lim_n\Hom_{W_n}(H^N(X,W_n\O_X(p^tD)),W_n)\right)_\mathbb Q.
            \end{aligned}
        \]
        This proves the second statement.

        For the third statement write $\F := \O_X(-D)$ and consider the derived push--forward of $\WO\tensor\Fu$:
        \[
            \begin{aligned}
                R\phi_*(W\Omega_X^N\tensor_{W\O_X}\Fu) 
                &\cong (F_S^t)_*(R\phi_*(W\Omega_X^N\tensor_{W\O_X}\Fu^{p^t})) & \text{(by Lem.~\ref{lemTanakaTwisting})} \\
                &\cong (F_S^t)_*R\phi_*(\lim_n(\WOn)\tensor_{W\O_X}\Fu^{p^t}) \\
                &\cong (F_S^t)_*R\lim_nR\phi_*(\WOn\tensor_{W_n\O_X}\Fun^{p^t}) & \text{(by Lem.~\ref{lemNToInfty})} \\
                &\cong (F_S^t)_*R\lim_n R\phi_*(R\sHom_{W_n\O_X}(\Fun^{-p^t}, \WOn)) & \text{(by Prop.~\ref{prop1})} \\
                &\cong (F_S^t)_*R\lim_n \sHom_{W_n\O_S}(R^N\phi_*\Fun^{-p^t}, W_n),
            \end{aligned}
        \]
        for large enough $t$, where the last isomorphism is again due to Ekedal~\cite{Eke}*{Theorem 4.1}.
        The third statement then follows analogously to the proof of the second statement.
    \end{proof}

    \begin{remark}
        While not the same, the proof of Proposition~\ref{propMiniDuality} is quite similar in spirit to those of \cite{Tanaka}. 
        One might therefore view it as a mere reformulation of his theorems from a duality--oriented viewpoint.
    \end{remark}

    We now attempt to establish a more general duality in the spirit of Ekedahl \cite{Eke}. 
    A crucial ingredient to Ekedahl's result was the isomorphism in $D(W[d])$:
    \[
        R_n\tensor^L_R R\Gamma_S(W\Omega_X^\bullet) \cong R\Gamma_S(W_n\Omega_X^\bullet).
    \]
    We will employ a similar property to our case.

    Define $\omega$ to be the Cartier-Dieudonn\'e ring $W_\sigma[F,V]$, that is the (non-commutative) $W$-algebra generated by $V$ and $F$, subject to the relations
    \[
        aV = V\sigma(a),
        Fa = \sigma(a)F \fa a \in W;
        VF = FV = p,
    \]
    where $\sigma$ is the Frobenius map on $W$, induced from that on $k$.
    While as a set $\omega$ is equal to $(\dsum_i WV^i) \oplus (\dsum_j WF^j)$, it is a non-commutative ring with an evident left--$W$--module structure.
    It follows from the definition (and the fact that $k^p = k$) that every element of $\omega$ can be uniquely described by a sum 
    \[
        \sum_{0<i} a_{-i}V^i + \sum_{0 \leq j} b_jF^j, a_i, b_j \in W.
    \]
    Let 
    \[
        \omega_n := \omega / V^n\omega,
    \]
    which is a ($W,\omega$)--bimodule, since $V^n\omega$ is a sub--left--$W$--module of $\omega$ and a right-$\omega$--ideal generated by $V^n$. 
    We then have, as sets, 
    \[
        \omega_n = \dsum_{0<i<n}a_{-i}V^i \oplus \dsum_{0\leq j}b_jF^j, a_{-i} \in W_{n-i}, b_j \in W_n.
    \]
    This yields two sets of left--$\omega$--module homomorphisms: 
    an obvious restriction map $\omega_n \xrightarrow \pi \omega_{n-1}$, as well as an injective map $\omega_{n-1} \xrightarrow \varrho \omega_{n}$, both induced by the respective maps $R$ and $\varrho = \{$multiplication by $p$\} on $W_\bullet$. 

    \begin{lemma}\label{lemTruncMod}
        Let $A$ be a $k$-algebra.
        Then $W(A)$ has a natural structure of left--$\omega$--modules and there is an isomorphism of left--$W$--modules
        \[
            \omega_n \tensor^L_\omega W(A)
            \cong W_n(A).
        \]
        For a sheaf of left--$\omega$--modules $M$ on $X$, 
        \[
            \omega_n \tensor^L_\omega R\Gamma(M)
            \cong R\Gamma(M_n),
        \]
        where $M_n := M / V^n M \cong \omega_n \tensor_\omega M$.
    \end{lemma}
    \begin{proof}
        The left--$\omega$--module structure on $W(A)$ is given by
        \[
            \begin{tikzcd}[row sep=tiny]
                \omega \times W(A) \arrow{r}{\cdot} & W(A)\\
                (\Sigma_i a_iV^i + \Sigma_j b_jF^j, w) \arrow[r, mapsto] 
                    & \Sigma_i a_iV^i(w) + \Sigma_j b_jF^j(w). \\
            \end{tikzcd}
        \]
        To compute the derived tensor product
        \[
            D(\omega-\mathfrak{lmod}) \xrightarrow{\omega_n\tensor^L_\omega \cdot} D(\mathfrak {ab}),
        \]take a projective resolution $P^\bullet$ of $\omega_n$:
        \[
            0 \rightarrow \omega \xrightarrow {V^n\cdot} \omega \rightarrow \omega_n \rightarrow 0.
        \]
        This complex of right--$\omega$--modules, when tensored with $W(A)$, yields a complex $P^\bullet\otimes_\omega W(A)$:
        \[
            0 \rightarrow W(A) \xrightarrow{V^n} W(A) \rightarrow 0.
        \]
        To see that this represents $W_n(A)$ simply observe that the map induced by $\omega \xrightarrow{V^n\cdot} \omega$ via the tensor product is precisely the $n$-fold Verschiebungs-map on $W(A)$:
        \[
            \begin{tikzcd}[row sep=tiny]
                W(A) \arrow{r}{\sim}
                & \omega \tensor_\omega W(A) \arrow{r}{V} 
                & \omega \tensor_\omega W(A) \arrow{r}{\sim} 
                & W(A) \\
                a \arrow[r, mapsto]
                & 1 \tensor a \arrow[r, mapsto] 
                & V \tensor a \arrow[r, mapsto]
                & V \cdot a = V(a).
            \end{tikzcd}
        \]
        Analogously, the action $F\cdot$ on $W(A)$ induced via the tensor product is the familiar Frobenius map $F$.

        Moreover, since $\omega_n$ is a left--$W$--module, so is $\omega_n\otimes_\omega^L W(A)$.
        Lastly, to see that the $D(\mathfrak{ab})$--isomorphism is in fact in $D(W-\mathfrak{lmod})$, simply observe that the left--$W$--module structures on both sides coincide via the isomorphism.

        For the second statement, let $M \in (X, \omega)$, that is $M$ is a sheaf of left--$\omega$--modules on $X$. 
        Let $P^\bullet$ be the projetive resolution of $\omega_n$
        \[
            0 \rightarrow \omega \xrightarrow {V^n \cdot} \omega  \rightarrow \omega_n \rightarrow 0.
        \]
        Then, since $P^i$ is flat for all $i$, 
        \[
            \omega_n\otimes^L_\omega R\Gamma(M) 
			\cong P^\bullet \tensor_\omega R\Gamma(M)
			\cong R\Gamma(P^\bullet\tensor_\omega M)
            \cong R\Gamma(\omega_n\otimes^L_\omega M)
            \cong R\Gamma(M_n).
        \]
    \end{proof}

    \begin{lemma}\label{lemOmegaModule}
        The ring $\omega$ has a natural $\mathbb Z$--grading given by $F$ and $V$:
        \[
            \omega = \left( \dsum_{0<i} WV^{i} \right) \oplus \left( \dsum_{0\leq j} WF^{j} \right).
        \]
        Let $D$ be a $\mathbb Q$--Cartier divisor on $X$, and write $\omega(D) := \dsum_{t\in \mathbb Z} W\O_X(p^tD)$. Then $\omega(D)$ is a sheaf of graded left--$\omega$--modules, and 
        \[
            \omega_n \tensor^L_\omega \omega(D) \cong \omega_n(D) := \dsum_{t\in\mathbb Z} W_n\O_X(p^tD).
        \]
        By Lemma~\ref{lemTruncMod} then
        \[
            \omega_n \tensor_\omega^L R\Gamma_X(\omega(D))
            \cong R\Gamma_X(\omega_n(D)).
        \]
    \end{lemma}
    \begin{proof}
        We have the following maps $F$ and $V$:
        \[
            \begin{aligned}
                W\O_X(D) \xrightarrow F F_* W\O_X(pD) \\
                F_*W\O_X(pD) \xrightarrow V W\O_X(D).
            \end{aligned}
        \]        
        That is by definition, 
        since 
        \[
            V^n(W\O_X(p^tD)) \subset W\O_X(p^{t-n}) \text{ and }
            F^n(W\O_X(p^tD)) \subset W\O_X(p^{t+n}D),
        \]
        $\omega(D)$ is in fact a sheaf of $\mathbb Z$--graded left--$\omega$--modules. The last statement follows from Lemma~\ref{lemTruncMod} and the fact that $W_n\O_X(D) \cong W\O_X(D) / V^n((F^n)_*W\O_X(p^nD))$.
    \end{proof}

    \begin{proposition}\label{propOmegaCheck}
        As left--$W$--modules, $\omega \cong ( \dsum_i W ) \oplus ( \dsum_j W )$. 
        Similarly, as left--$W_n$--modules, $\omega_n \cong ( \dsum_{i<n}W_{n-i} ) \oplus (\dsum_j W_n)$.
        It follows that 
        \[
            \Hom_{W_n}(\omega_n, W_n) \cong \left( \dsum_{0<i<n}W_{n-i} \right) \oplus \left( \prod_{0\leq j} W_n \right)
        \]
        as $W_n$--modules.
        With the left--$\omega$--module structure induced by the right--structure on $\omega_n$, there is an isomorphism
        \[
            \Hom_{W_n}(\omega_n, W_n) 
            \cong 
            \dsum_{0<i<n} F^{i}  W_{n-i}
            \oplus \prod_{0\leq j} F^{-j}W_n
        \]
        of left--$\omega$--modules.
    \end{proposition}
    \begin{proof}
        Since $k=k^p$, elements $a \in \omega$ can be uniquely written as 
        \[
            a = \sum_i a_iV^i + \sum_j b_jF^j, 
            a_i, b_j \in W.
        \]
        The natural identification is clearly additive and bijective:
        \[
            \begin{tikzcd}
                & W \arrow{dl} \arrow{dr} \\
                \omega \arrow{rr}{\sim} && (\dsum_i W)\oplus( \dsum_j W) \\
                \sum_i a_iV^i + \sum_jb_jF^j \arrow[rr, mapsto] && \sum_i a_i + \sum_jb_j.
            \end{tikzcd}
        \]
        It is $W$--linear (on the left), since the left--$W$--module structure of $\omega$ is simply multiplication on the left.
        Analogously, $\omega_n \cong ( \dsum_{i<n}W_{n-i}) \oplus (\dsum_j W_n)$ as left--$W$--modules.
        Therefore in $(W-\mathfrak{mod})$,
        \begin{equation}
            \begin{aligned}
                \Hom_{W_n}(\omega_n, W_n) 
                &\cong \left(\dsum_{0<i<n} \Hom_{W_n}(W_{n-i}, W_n)\right) \\
                &\oplus \left( \prod_{0\leq j} \Hom_{W_n}(W_n, W_n) \right) \\
                &\cong \left( \dsum_{0<i<n} W_{n-i} \right) \oplus \left( \prod_{0\leq j} W_n \right) =: \check\omega_n.
            \end{aligned}
        \end{equation}
        
        The right--$\omega$--module structure on $\omega_n$ induces a structure of (graded) left--$\omega$--modules on $\Hom_{W_n}(\omega_n,W_n)$. 
        For any $\alpha \in \check\omega_n$, 
        \[
            \begin{aligned}
                V\cdot (\omega_n \xrightarrow {\alpha} W_n ) 
                    &= \omega_n \xrightarrow {\alpha \circ (\cdot V)} W_n, \\
                F\cdot (\omega_n \xrightarrow {\alpha} W_n ) 
                    &= \omega_n \xrightarrow {\alpha \circ (\cdot F)} W_n.
            \end{aligned}
        \]
        Let
        \[
            \begin{aligned}
                \alpha &= \dsum_{i} \alpha_i \in \dsum_{0\leq i<n} \Hom_{W_n}(W_{n-i}V^i,W_n) \cong \dsum_{0\leq i<n} W_{n-i}, \\
                \beta &= \prod_j \beta_j \in \prod_{0\leq j} \Hom_{W_n}(W_nF^j,W_n) \cong \prod_{0\leq j}W_n.
            \end{aligned}
        \]
        Under the above isomorphisms, $\alpha_i \in W_{n-i}$ corresponds to the map in $\Hom(W_{n-i},W_n)$ which takes $1$ to that element in $W_n$ which corresponds to $\alpha_i$ under the isomorphism $W_{n-i} \xrightarrow {p^i} \text{im}(p^i) \subset W_n$. 
        That is, it takes $1$ to $p^i\alpha_i \in W_n$. 
        Similarly, $\beta_j \in W_n$ corresponds to the map in $\Hom(W_n,W_n)$ which takes $1$ to $\beta_j$.

        Let $\alpha_k = 0 \fa k \neq i, \beta_l = 0 \fa l \neq j$. 
        That is $\alpha$ is zero outside of $W_{n-i}V^i \subset \omega_n$, and $\beta$ is zero outside of $W_nF^j \subset \omega_n$.
        Then 
        \[
            \begin{aligned}
                V\cdot\alpha &= p^{-i+1}((V\cdot\alpha_i)(1)) = p^{-i+1}((\alpha_{i}\circ(\cdot V))(1)) \\
                    &= p^{-i+1}p^{i}\alpha_{i} = p\alpha_{i} \in W_{n-i+1} \fa 0 \leq i, \\
                V\cdot\beta &= (V\cdot\beta_j)(1) = (\beta_{j}\circ (\cdot V))(1) \\
                    &= \beta_{j}(p) = p\beta_j \in W_n \fa 0 < j; \\
                F\cdot\alpha &= p^{-i-1}((F\cdot\alpha_i)(1)) = p^{-i-1}((\alpha_{i}\circ(\cdot F))(1)) \\
                    &= p^{-i-1}(\alpha_{i}(p)) = p^{-i-1}(pp^i\alpha_{i}) = R(\alpha_i) \in W_{n-i-1} \fa 0 \leq i \\
                F\cdot\beta &= (F\cdot\beta_j)(1) = (\beta_j \circ(\cdot F))(1) \\
                    &= \beta_j(1) = \beta_j \in W_n \fa 0 < j,
            \end{aligned}
        \]
        where $R$ is the natural restriction map $W_k \xrightarrow R W_{k-1}$.
        There is thus an isomorphism of left--$\omega$--modules
        \[
            \begin{aligned}
                \check\omega_n 
                &\cong 
                \dsum_{0<i<n} F^{i}  W_{n-i}
                \oplus \prod_{0\leq j} F^{-j}W_n \\
                &\cong
                \dsum_{0<i<n} F^iW_{n-i}
                \oplus \prod_{0\leq j} V^jp^{-j}W_n,
            \end{aligned}
        \]
        where the graded left--$\omega$--module structure on the latter is the natural one, that is multiplication on the left.

    \end{proof}

    Note that the injective left--$W$--linear maps $\omega_{n-1} \xrightarrow \varrho \omega_n$ form a direct system. 
    $\Hom_{W_n}(\omega_n, W_n[-N])$ then form an inverse system (cf. \cite{Eke}*{III.2.3.*}) with boundary maps $\pi$ defined by the commutativity of the diagram
    \begin{equation}\label{diagIndPi}
        \begin{tikzcd}
            \sHom_{W_n}(\omega_n, W_n[-N]) \arrow{r}{\varrho^*} \arrow{d}{\pi} & \sHom_{W_n}(j_{n,*}\omega_{n-1}, W_n[-N]) \\
            j_{n,*}\sHom_{W_{n-1}}(\omega_{n-1}, W_{n-1}[-N]) \arrow{ur}{\varrho_*}.
        \end{tikzcd}
    \end{equation}
    Here $W_{n-1}S \xrightarrow {j_n} W_nS$ is the natural immersion.
    There exist unique such maps $\pi$ because, by coherent duality and the fact that $W_{n-1} \cong j_n^!W_n$, $\varrho_*$ is an isomorphism. 

    We can now prove the main duality theorem.
    \begin{theorem}\label{thmDuality}
        Let $X$ be a smooth projective variety over a perfect field $k$ of characteristic $p>0$, and $D$ be a $\mathbb Q$--Cartier divisor on $X$.
        Write 
        \[
            \check\omega := 
            \prod_{j\in\mathbb Z} F^{j}W
        \]
        Then
        \[
            \begin{aligned}
                &\prod_{t\in\mathbb Z} R\phi_*R\lim_nR\sHom_{W_n\O_X}(W_n\O_X(p^tD),\WOn) \\
                &\cong R\Hom_\omega\left( R\phi_* \omega(D), \check\omega[-N]\right).
            \end{aligned}
        \]
    \end{theorem}
    \begin{proof}
        \[
            \begin{aligned}
                &\prod_{t\in\mathbb Z} R\phi_* R\lim_nR\sHom_{W_n\O_X}(W_n\O_X(p^tD),\WOn) \\
                &\cong R\lim_n\prod_{t\in\mathbb Z} R\phi_* R\sHom_{W_n\O_X}(W_n\O_X(p^tD),\WOn) & \text{(e.g. \cite{stacks}*{0BKP})}\\
                &\cong R\lim_n R\sHom_{W_n\O_S}\left(\dsum_{t\in\mathbb Z} R\phi_*W_n\O_X(p^tD), W_n[-N]\right) & \text{(by \cite{Eke}*{Thm. 4.1})} \\
                &\cong \lim_n R\Hom_{W_n}\left(\omega_n \tensor^L_\omega R\Gamma_X(\omega(D)), W_n[-N]\right). & \text{(by Lem.~\ref{lemOmegaModule} and Prop.~\ref{propMiniDuality})}
            \end{aligned}
        \]

        By derived tensor--hom adjunction 
        (see for instance~\cite{Yek}*{Proposition 14.3.18}) 
        we have an isomorphism 
        \begin{equation}\label{eqCohDuality}
            \begin{aligned}
                &R\lim_n R\Hom_{W_n}\left(\omega_n \tensor^L_\omega R\Gamma_X(\omega(D)), W_n[-N]\right) \\
                &\cong R\lim_n R\Hom_{\omega}\left(R\Gamma_X(\omega(D)), R\Hom_{W_n}(\omega_n, W_n[-N])\right)
            \end{aligned}
        \end{equation}
        in $D(W-\mathfrak{bimod})$.

        Let $0 \neq w \in \Hom(W_{n-i}, W_n) \cong W_{n-i}$. 
        Since $\pi$ is induced by the commutative Diagram~\ref{diagIndPi}, $\pi(w)$ is the unique map such that the following diagram commutes:

        \[
            \begin{tikzcd}
                W_{n-i} \arrow{r}{w} & W_n \\
                W_{n-i-1} \arrow{u}{\varrho} \arrow{r}{\pi(w)} & W_{n-1} \arrow{u}{\varrho}
            \end{tikzcd}.
        \]
		This unique map is $R(w)$ (since for $\tau\in W_{n-i-1}, \varrho(R(w)(\tau)) = w\varrho(\tau) \in W_n$).		
        The induced maps $W_n \xrightarrow \pi W_{n-1}$ are therefore precisely the term-wise restriction maps $R$.
        Taking the limit yields an isomorphism of right--$W_n$--modules
        \[
            \begin{aligned}
                \lim_n\Hom_{W_n}(\omega_n, W_n) 
                &\cong \lim_n \left\{ \left( \dsum_{0<i<n} W_{n-i} \right) \oplus \left( \prod_{0\leq j} W_n \right) \right\} \\
                &\cong\prod_{j\in\mathbb Z}W,
            \end{aligned}
        \]
        whose $\omega$--module structure is given by that on the $\check\omega_n$ (cf. Proposition~\ref{propOmegaCheck}).
        That is as sets 
        \[
            \lim_n\Hom_{W_n}(\omega_n,W_n)
            \cong 
            \prod_{j\in\mathbb Z} F^{-j}W
            =: \check\omega.
        \]
        with the obvious structure of graded left--$\omega$--modules.
        The result then follows from Equation \ref{eqCohDuality}.
    \end{proof}

\subsection{Computation and application to vanishing}
    In this section we consider divisors $D$ such that 
    \[
        \begin{aligned}
            R^j\Gamma_X(W\O_X(p^tD)) &= 0 \fa j<N, \text{ large enough $t$}, \\
            R^N\Gamma_X(W\O_X(p^tD)) &= \text{torsion--free for large enough $t$}.
        \end{aligned}
    \]
    In particular this is the case for $D$ such that $-D$ is ample (cf.~\cite{Tanaka}*{Theorem 4.14 and 5.3, Step 5}).
    For such $D$ the following finiteness lemma holds.
    
    \begin{lemma}\label{lemFiniteness}
        \[
            R^j\Gamma_X(W\O_X(p^tD)) \cong R^j\Gamma_X(W_n\O_X(p^tD))
        \]
        for any $j<N, n_0 \leq n$, where $n_0$ depends on $t$, for any $t$.
    \end{lemma}
    \begin{proof}
        By assumption, for large enough $n$ the short exact sequence of $W\O_X$--modules 
        \[
            0 \rightarrow F^n_*W\O_X(p^{t+n}D) \xrightarrow {V^n} W\O_X(p^tD)
            \xrightarrow R W_n\O_X(p^tD) \rightarrow 0
        \]
        induces exact sequences of $W$--modules
        \[
            0 \rightarrow 0 
            \xrightarrow {V^n} R^j\Gamma_X(W\O_X(p^tD))
            \xrightarrow R R^j\Gamma_X(W_n\O_X(p^tD)) \rightarrow 0
        \]
        for all $j < N$.
    \end{proof}
    
    The following corollary shows that the dual's cohomology's vanishing depends on the $V$--torsion of the cohomology.

    \begin{corollary}\label{corTanVan}
        Under the above assumptions, 
        \[
            \begin{aligned}
                &R^i\phi_* \sHom_{W\O_X}(W\O_X(p^tD), \WO)) \\
                \cong~&R^{i-2}\Gamma_X\left(R^1\lim_n \sHom_{W\O_X}(W\O_X(p^tD), \WOn)\right)
            \end{aligned}
        \]
        $\fa 0<i, t \in \mathbb Z.$
        
        If $p^tD$ is $\mathbb Z$--Cartier or $i = 1$, this is equal to zero.
        Otherwise it is torsion.
    \end{corollary}
    \begin{proof}
        First we will show that 
        \begin{equation}\label{eqCorStep1}
            \begin{aligned}
                &\prod_{t\in\mathbb Z}R^i\phi_*R\lim_nR\sHom_{W_n\O_X}(W_n\O_X(p^tD),\WOn) \\
                &\cong h^i\left(\lim_n R\Hom_{W_n}\left(\omega_n \tensor^L_\omega R\Gamma_X(\omega(D)), W_n[-N]\right)\right) = 0 \fa 0 < i.
            \end{aligned}
        \end{equation}

        There is a spectral sequence
        \[
            E_{n,2}^{p,q} := \omega_n \otimes^{L^p}_\omega R^q\Gamma_X(\omega(D))
            \Rightarrow \omega_n \otimes^{L^{p+q}}_\omega R\Gamma_X(\omega(D))
            =: E_n^{p+q}.
        \]
        The sequence degenerates at the second sheet, and so we have an exact sequence 
        \[
            0 \rightarrow E_{n,2}^{0,j} \rightarrow E_{n}^{j} \rightarrow E_{n,2}^{-1,j+1}
            \rightarrow 0.
        \]
        We now consider the two outer terms $E_{n,2}^{0,j} \cong \omega_n\tensor R^t\Gamma(\omega(D))$  and $E_{n,2}^{-1,j+1} \cong R^{j+1}\Gamma(\omega(D))[V^n]$ (where $M[V^n]$ denotes the $V^n$--torsion of a left--$\omega$--module $M$) separately. 

        By tensor--hom adjunction 
        we have an isomorphism 
        \[
            \begin{aligned}
                &\lim_n \Hom_{W_n}\left(\omega_n \tensor_\omega R^{-i}\Gamma_X(\omega(D)), W_n[-N]\right) \\
                &\cong \lim_n \Hom_{\omega}\left(R^{-i}\Gamma_X(\omega(D)), \Hom_{W_n}(\omega_n, W_n[-N])\right) \\
                &\cong \lim_n \Hom_{\omega}\left(R^{N-i}\Gamma_X(\omega(D)), \check\omega_n)\right) \\
                &\cong \Hom_\omega\left( R^{N-i}\Gamma_X(\omega(D)), \check\omega)\right)
                = 0 \fa 0<i
            \end{aligned}
        \]
        in $D(W-\mathfrak{bimod})$, where the final equality follows from the assumption on torsion.

        Consider now the right outer term.
        \[
            \begin{aligned}
                &\lim_n \Hom_{W_n}\left(
                \text{ker}\left(R^{N-i+1}\Gamma_X(\omega(D)) \xrightarrow {V^n} R^{N-i+1}\Gamma_X(\omega(D)\right), 
                W_n\right) \\
                & \cong \prod_t \lim_n \Hom_{W_n}\left(
                \text{ker}\left(R^{N-i+1}\Gamma_X(W\O_X(p^{t+n}D)) \xrightarrow {V^n} R^{N-i+1}\Gamma_X(W\O_X(p^tD)\right), 
                W_n\right).
            \end{aligned}
        \]
        $R^j\Gamma_X(W\O_X(p^tD))$ is torsion--free for large enough $t$ and any $j$.
        So only finitely many objects of each inverse system are non--zero, wherefore the inverse limit is also zero.

        Since $\lim\Hom(E_{n,2}^{0,j},W_n) = 0 = \lim\Hom(E_{n,2}^{-1,j+1},W_n) \fa j<N$, 
        \[
            \lim_n\Hom_{W_n}(E_n^j,W_n) = 0 \fa j<N.
        \]
        We have now shown that Equation~\ref{eqCorStep1} holds. 
        Unfortunately, we do not know whether 
        \[
            R^1\lim_n \sHom_{W_n\O_X}(W_n\O_X(D),\WOn) = 0 ,
        \]
        and thus whether
        \[
            \sHom_{W\O_X}(W\O_X(D),\WO) \cong R\lim_nR\sHom_{W_n\O_X}(W_n\O_X(D),\WOn)
        \]
        holds for $D$ not $\mathbb Z$--Cartier.
        In order to describe the left hand side using Theorem~\ref{thmDuality}, we therefore need to consider another spectral sequence.
        Write $\mathcal H_{n,t} := \sHom_{W\O_X}(W\O_X(p^tD),\WOn)$. 
        There is a spectral sequence
        \begin{equation}\label{eqSS}
            E_{t,2}^{p,q} := R^p\phi_*R^q\lim_n \mathcal H_{n,t} \Rightarrow R^{p+q}\phi_*R\lim_n \mathcal H_{n,t} =: E_t^{p+q}.
        \end{equation}
        Since $R^i\lim \mathcal H_{n,t} = 0 \fa 1<i$ (condition (1) of Lemma~\ref{lemCR} is satisfied), page two of the spectral sequence contains only two nonzero rows, $q=1$ and $q=0$. 
        Consequently it degenerates at page three, and 
        \[
            \text{Fil}^nE_t^n \cong E_{t,3}^{n,0} \cong E_{t,2}^{n,0}/d(E_{t,2}^{n-2,1}).
        \]
        We obtain a long exact sequence
        \begin{equation}\label{eqLES}
            \cdots \rightarrow E_{t,2}^{n-2,1} 
            \rightarrow E_{t,2}^{n,0} 
            \rightarrow E_t^n 
            \rightarrow E_{t,2}^{n-1,1}
            \rightarrow \cdots.
        \end{equation}
        We have $E_t^n = 0$ for $0<n$. 
        Therefore $E_{t,2}^{1,0} \cong E_{t,2}^{-1,1} = 0$, and $E_{t,2}^{i,0} \cong E_{t,2}^{i-2,1}$ for $1<i$.
        Due to the Twisting Lemma~\ref{lemTanakaTwisting}, 
        \[
            (E_{t,2}^{i,0})_\mathbb Q = (E_{t+s,2}^{i,0})_\mathbb Q = 0,
        \]
        for large enough $s \in \mathbb N$.
    \end{proof}


\begin{bibdiv}
\begin{biblist}
\bib{CR}{article}{
	author = {Chatzistamatiou, Andre},
	author = {Rülling, Kay},
	year = {2011},
	pages = {},
	title = {Hodge-Witt cohomology and Witt-rational singularities},
	volume = {17},
	journal = {Documenta Mathematica}
}

\bib{Eke}{article}{
	author={Ekedahl, Torsten},
	title={On the multiplicative properties of the de Rham---Witt complex. I},
	journal={Arkiv f{\"o}r Matematik},
	year={1984},
	volume={22},
	number={2},
	pages={185--239}
}

\bib{Illusie}{article}{
	author = {Illusie, Luc},
	journal = {Annales scientifiques de l'École Normale Supérieure},
	language = {fr},
	number = {4},
	pages = {501-661},
	publisher = {Elsevier},
	title = {Complexe de de Rham-Witt et cohomologie cristalline},
	url = {http://eudml.org/doc/82043},
	volume = {12},
	year = {1979}
}

\bib{stacks}{book}{
	author       = {The {Stacks project authors}},
	title        = {The Stacks project},
	url 		 = {https://stacks.math.columbia.edu},
	year         = {2019},
	label		 = {Stacks}
}

\bib{Tanaka}{article}{
	author = {Tanaka, Hiromu},
	title = {Vanishing theorems of Kodaira type for Witt Canonical sheaves},
	journal = {arXiv:1707.04036v3 [math.AG]},
	year = {2020}
}

\bib{Yek}{book}{
    place={Cambridge}, 
    series={Cambridge Studies in Advanced Mathematics}, 
    title={Derived Categories}, 
    DOI={10.1017/9781108292825}, 
    publisher={Cambridge University Press}, 
    author={Yekutieli, Amnon}, 
    year={2019}
}
\end{biblist}
\end{bibdiv}
\end{document}